\documentclass[11pt,a4paper]{amsart}
\usepackage{enumerate, comment}
\usepackage{amssymb}
\usepackage{amsmath}
\usepackage{latexsym}
\usepackage{amsthm}
\usepackage{xypic}
\usepackage[all,cmtip]{xy}
\usepackage{hyperref}
\usepackage{comment}

\newtheorem{theorem}{Theorem}

\newtheorem{proposition}{Proposition}

\theoremstyle{Definition}

\theoremstyle{remark}
\newtheorem{remark}[theorem]{Remark}

\theoremstyle{question}
\newtheorem{question}{Question}

\theoremstyle{corollary}

\numberwithin{equation}{section}

\usepackage[english]{babel}

\newcommand{\Q}{\mathbb{Q}}

\numberwithin{equation}{section}

\begin{document}

\newcommand{\rx}{\mathcal{R}(X)}

\title{A note on Galois groups and local degrees}
\author{Sara Checcoli}
\address{Institut Fourier\\ 100 rue des Maths, 
St. Martin d'Heres, France}\email{sara.checcoli@univ-grenoble-alpes.fr}

\begin{abstract}
In this paper, we consider infinite Galois extensions of number fields and study the relation between their local degrees and the structure of their Galois groups.
It is known that, if $K$ is a number field and $L/K$ is an infinite Galois extension of group $G$, then the local degrees of $L$ are uniformly bounded at all rational primes if and only if $G$ has finite exponent. In this note we show that the non uniform boundedness of the local degrees is not equivalent to any group theoretical property. More precisely, we exhibit several groups that admit two different realisations over a given number field, one with bounded local degrees at a given set of primes and one with infinite local degrees at the same primes.
\end{abstract}

\maketitle

\section{Introduction}
This note is about the characterisation of infinite Galois  extensions of number fields with special local behaviours in terms of their Galois group. 

We say that a subfield  $F$ of ${\overline{\mathbb{Q}}}$ has \emph{bounded local degrees at a prime number $p$} if it can be embedded in some finite extension of $\Q_p$. We say that $F$ has \emph{uniformly bounded local degrees} if there exists a positive integer $n$ such that, for every prime $p$, $F$ can be embedded in some finite extension of $\Q_p$ of degree at most $n$.

Clearly all number fields have uniformly bounded local degrees. 
An example of an infinite extension of $\Q$ with uniformly bounded local degrees, given by Bombieri and Zannier in \cite [Proposition 1]{BZ}, is  the compositum of all extensions of degree at most $d$ of a number field. 

In \cite{CZ} and \cite{C} it is proved that a Galois extension of a number field has uniformly bounded local degrees if and only if its Galois group has finite exponent. So, for a Galois extension, the  uniform boundedness of the local degrees translates into a property of its Galois group. The aim of this paper is to investigate whether the same happens for Galois extensions with bounded local degrees at some primes. More precisely:
\begin{question}\label{Q1}
 Let $\mathcal{S}$ be a set of rational primes, $K$ a number field and $L/K$ a Galois extension of group $G$. Does there exist a group theoretical property on $G$ equivalent to the boundedness of the local degrees of $L/K$ at all primes in $\mathcal{S}$?
\end{question}
This question has some interest also in connection to the \emph{Bogomolov property (B)}, a property on the height of algebraic numbers, introduced by Bombieri and Zannier in \cite{BZ}: a subfield $L$ of $\overline{\Q}$ has \emph{property (B)} if the absolute logarithmic Weil height of every element in $L$ is either zero or is bigger than an absolute positive constant. 

Northcott's theorem easily implies that every number field satisfies property (B), while to determine whether property (B) holds  for an infinite algebraic extension of $\Q$ is in general a difficult problem. Nevertheless, property (B) was proved to be true for several fields, including the maximal totally real extension $\Q^{tr}$ of $\Q$ (by Schinzel in \cite{Sch}),  the maximal abelian extension of a number field $K$ (by Amoroso and Zannier, see \cite{AZ1} and \cite{AZ2}, generalising a former result of  Amoroso and Dvornicich in \cite{AD} for $K=\Q$) and the field generated over $\Q$ by all the torsion points of an elliptic curve defined over $\Q$ (by Habegger in \cite{Hab}).

A connection between the boundedness of the local degrees and property (B) was established in \cite{BZ}, where the authors prove that every Galois extension $L/K$ of a number field $K$ having bounded local degrees at some prime  enjoys  property (B) (so, in particular, the field of totally $p$-adic numbers $\Q^{tp}$ has (B), yielding a $p$-dic analogue of Schinzel's result). This was later generalised by Amoroso, David and Zannier in \cite{ADZ} where it is proved that, if $G$ is the Galois group of $L/K$, property (B) holds for $L$ under the weaker assumption that the subfield of $L$ fixed by the center of $G$ has bounded local degrees at some prime.

In \cite{ADZ} the authors also define property (B) for groups:
a profinite group $G$ has \emph{property (B)} if for every number field $K$ and for every Galois extension $L/K$ of group $G$, the field $L$ has property (B).

To the best of the author's knowledge, all the known examples of groups having property (B) are groups $G$ such that the quotient of $G$ by its center has finite exponent: indeed, in this case, if $L/K$ is a Galois extension of group $G$, the subfield of $L$ fixed by the center of $G$ will have \emph{uniformly} bounded local degrees at \emph{all} primes. It is natural to investigate whether other groups with (B) exist and what can be said about their structure.  In this respect, Question \ref{Q1} asks whether it is possible to characterise those groups having (B) because of the (non uniform) boundedness of the local degrees at some primes of their realisations over any number field. We recall that if $G$ is a group and $K$ is a number field a \emph{realisation of $G$ over $K$} is a Galois extension $L/K$ of group $G$. 

In this paper we give a negative answer to Question \ref{Q1}, even when the local degrees are bounded at all rational primes, by proving the following result:
\begin{theorem}\label{T1}
Let $\mathcal{S}$ be a set of rational primes. Let $K$ be a number field and let $\mu(K)$ be the group of its roots of unity. 
Let $G=\prod_{m\geq 1} G_m$ be a direct product of a family of finite groups $\{G_m\}_m$ with unbounded exponents and suppose that, for every $m$, the group $G_m$ satisfies one of the following conditions:
\begin{enumerate}
\item\label{l1} $G_m$ is abelian of odd order;
\item\label{l2} $G_m$ is solvable of order prime to $|\mu(K)|$;
\item\label{l3} $G_m$ is an iterated semidirect products of abelian groups of order not divisible by the primes in $\mathcal{S}$.
\end{enumerate}
Then $G$ has a realisation over $K$ with bounded local degrees at all primes in $\mathcal{S}$. If moreover the groups $G_m$'s have coprime orders, then $G$ admits also a realisation over $K$ with unbounded local degrees at all primes in $\mathcal{S}$.
\end{theorem}
The proof of Theorem \ref{T1} is based on the existence of solutions to certain \emph{Grunwald problems}. The Grunwald problem, a strong variant of the inverse Galois problem, asks whether,  given a finite group $G$ and a number field $K$, there exists a realisation of $G$ over $K$ having prescribed completions at a given finite set of primes. In Section \ref{Section1}, Theorem \ref{main1}, we prove that the direct product of the groups $G_m$ can be realised in two ways (with bounded or unbounded local degrees at any fixed set of primes) whenever every \emph{cyclic} Grunwald problems for $G_m$ has \emph{enough} solutions. Theorem \ref{T1} is then a corollary of Theorem \ref{main1} and results of  Grunwald, Wang, Neukirch, Demarche, Lucchini Arteche and Neftin. At the end of Section \ref{Section1} the link between our problem and property (B) for groups is discussed in more details.

The proof of Theorem \ref{T1} does not produce explicit realisations. A more explicit construction for certain abelian groups is given in Section \ref{Section3}, where a link between the boundedness of the local degrees  of a field and the degree of its finite monogenic subextensions is also discussed.

In Section \ref{Section4}, Proposition \ref{t31}, we show that the existence of a realisation with bounded local degrees at some primes holds also for other classes of profinite groups than direct products: these groups are constructed as  iterated semidirect products by finite nilpotent groups. 
The proof relies on the existence of solutions to certain embedding problems with nilpotent kernel (a central step in the proof of Shafarevich's theorem on the realisability of finite solvable groups). This, however, produces extensions with a quite special local structure and this approach seems to be difficult to exploit for more general families of profinite groups.

\section{The Grunwald problem and proof of Theorem \ref{T1}}\label{Section1}
Let $K$ be a number field. If $v$ is a prime of $K$, we denote by $K_{v}$ the completion of $K$ with respect to $v$. Let $G$ be a finite group and $\mathcal{T}$ a finite set of primes of $K$. The \emph{Grunwald problem} asks whether for every choice of finite Galois extensions $L_v$ of $K_v$, $v\in\mathcal{T}$, of group $G_v$ embeddable in $G$, there exists a finite Galois extension $L/K$ of group $G$ having completions isomorphic to $L_v$ at the primes above  $v$, for all $v\in\mathcal{T}$. 

When the answer is affirmative, we call $L$ a \emph{solution of the Grunwald problem for $(G,\mathcal{T},\{G_v\}_{v\in\mathcal{T}})$}.

The Grunwald problem is shown to have a positive answer, for example, when $G$ is an abelian group of odd order by the Grunwald-Wang theorem (see\cite{Gru}, \cite{Wan}), $G$ is a solvable group of order prime to the order of the group of roots of unity in $K$ (by Neukirch \cite{Neuk}), $G$ has a generic extension over $K$ (this means that all extensions of $K$ of group $G$ are obtained by specialisations of the same function field extension, see Saltman's work \cite{Sal}), $G$ is regular over $K$, the local extensions are unramified and the set $\mathcal{T}$ only contains primes bigger than a certain constant depending on $G$ (by D\`ebes and Ghazi \cite{DG}), $G$  is an iterated product of abelian groups and the primes in $\mathcal{T}$ do not divide the order of $G$ (by Demarche, Lucchini Arteche and Neftin \cite{DLAN}).
Conjecturally (see Colliot-Th\'el\`ene's paper \cite[Section 2]{CT} and also \cite[Section 2.5]{DLAN})  for every group $G$, there is a set of 'bad' primes such that if $\mathcal{T}$ is disjoint from this set, then the Grunwald problem has a positive answer.

To prove Theorem \ref{T1}, we need to construct realisations of direct products of groups with special local behaviours. The construction relies on the existence of solutions to certain cyclic Grunwald problems and on the fact that such solutions can be chosen to  be linearly disjoint. More precisely, we have the following:
\begin{theorem}\label{main1}
Let $\mathcal{S}=\{p_1, p_2, p_3,\ldots\}$ be a  set of rational primes and let $K$ be a number field. Let $\mathcal{T}=\{v_1,v_2,v_3,\ldots\}$ be a set of primes of $K$ such that $v_i$ lies above $p_i$ for every $i$ and set $\mathcal{T}_m=\{v_1,\ldots\,v_m\}$. 
Let $\{G_m\}_m$ be a family of finite groups with unbounded exponents and denote by $G=\prod_{m\geq 1} G_m$ their direct product. 

Suppose that, for every choice of a family $\{\{G_{m,v}\}_{v\in\mathcal{T}_m}\}_m$ where, for every $m$ and for every $v\in\mathcal{T}_m$, $G_{m,v}$ is a cyclic group  embeddable in $G_m$, there exists a family $\{L_m\}_m$ of Galois extensions of $K$ such that, for every $m$, $L_{m}/K$ is a solution to the Grunwald problem for $(G_m,\mathcal{T}_m, \{G_{m,v}\}_{v\in \mathcal{T}_m})$ and $L_m$ is linearly disjoint, over $K$, from the compositum of the fields $L_1,\ldots, L_{m-1}$.

Then $G$ admits two realisations over $K$, one with local degrees at all primes $p_i\in\mathcal{S}$  bounded by $[K:\Q]$ if $i=1$ and by $[K:\Q]\prod_{j=1}^{i-1} |G_j|$ if $i\geq 2$, and one with unbounded local degrees at all primes in $\mathcal{S}$.
\end{theorem}

\begin{proof}

\emph{Realisation with bounded local degrees.}
For every $m\geq 1$ and  for every $v\in\mathcal{T}_m$, take $G_{m,v}=1$ equal to the identity group. By hypothesis, there exists a family of extensions of $K$, $\{L_m\}_m$, such that $L_m/K$ is a solution of the Grunwald problem $(G_m,\mathcal{T}_m, \{1\}_{v\in \mathcal{T}_m})$ and $L_m$ is linearly disjoint over $K$ from the compositum of the fields $L_1,\ldots, L_{m-1}$.

Denote by $L$ the compositum of all $L_m$'s. Then $L$ is a Galois extension of $K$ of group $G$ and its local degrees at $p_i\in \mathcal{S}$ are bounded by $[K:\Q]$ if $i=1$ and by $[K:\Q]\prod_{j=1}^{i-1} |G_j|$ if $i\geq 2$. 

\smallskip

\emph{Realisation with unbounded local degrees.} As the exponents of the groups $G_m$ are unbounded, for every  integer $n\geq 2$, there exists an index $m_n\neq m_{n-1}$ (with $m_1=1$) such that the group  $G_{m_n}$ possesses a cyclic subgroup $\mathcal{C}_n$ of order  bigger than $n$. Moreover the $m_n$'s can be chosen so that $m_n<m_{n+1}$ for every $n\geq 1$. We denote by $\mathcal{I}=\{m_n\}_{n\geq 1}$ the set of such indexes (clearly the set $\mathcal{I}$ is infinite and unbounded).

Now, for every $m\geq 1$ and for every $v\in \mathcal{T}_m$, set $G_{m,v}=1$ if $m\not\in \mathcal{I}$ and $G_{m,v}=\mathcal{C}_n$ if $m=m_n\in\mathcal{I}$.
By hypothesis, there exists a family $\{L_m\}_m$ of Galois extensions of $K$ such that, for every $m\geq 1$, $L_m/K$ is a solution of the Grunwald problem $(G_m,\mathcal{T}_m, \{G_{m,v}\}_{v\in \mathcal{T}})$ and $L_m$ is linearly disjoint over $K$ from the compositum of the fields $L_1,\ldots, L_{m-1}$.

Denoting by $L$ the compositum of all $L_m$'s, we have that $L$ is a Galois extension of $K$ of group $G$ and its local degrees at any prime $p_i\in\mathcal{S}$ are unbounded as, for all $m_n\in\mathcal{I}$ with $m_n\geq i$, the local degrees of $L_{m_n}/K$ at $p_i$ are bigger than $n$.
\end{proof}

Now Theorem \ref{T1} follows as a corollary of Theorem \ref{main1} and the above cited results.
\begin{proof}[Proof of Theorem \ref{T1}]
Let $\{G_m\}_m$ be a family of finite groups such that, for every $m$, $G_m$ satisfies one of the three conditions \eqref{l1}, \eqref{l2} or \eqref{l3} of Theorem \ref{T1}. Then, by the above cited results, the Grunwald problem has a positive answer for $G_m$: more precisely, this follows from the Grunwald-Wang theorem (see \cite{Gru}, \cite{Wan}) in case \eqref{l1}, from Neukirch's result \cite{Neuk} in case \eqref{l2} and from the recent work of Demarche, Lucchini Arteche and Neftin \cite{DLAN} in case \eqref{l3}.

Notice that if the groups $G_m$'s have coprime orders, then all their realisations will be automatically linearly disjoint and we can conclude by Theorem \ref{main1}. 

Assume now that the orders of the groups $G_m$'s are not coprime. We want to show that also in this case we can find a realisation of the group $G=\prod_{m\geq 1} G_m$ over $K$ having bounded local degrees at all primes in $\mathcal{S}$.
Notice that if $\Gamma$ is a finite group which satisfies one of the conditions \eqref{l1}, \eqref{l2} or \eqref{l3} of Theorem \ref{T1}, then the direct product $\Gamma^n$ of $n$ copies of $\Gamma$ will satisfy the same condition for every integer $n\geq 1$.
Indeed, this is clearly true when $\Gamma$ satisfies \eqref{l1} of \eqref{l2}. We now want to show that if $\Gamma_1$ and $\Gamma_2$ are two groups satisfying condition \eqref{l3}, then $\Gamma1\times \Gamma_2$ also satisfies condition \eqref{l3}. By hypothesis, $\Gamma_1$ can be written as $H_1\rtimes_{\varphi}A_1$ where $A_1$ is abelian, $H_1$ satisfies property \eqref{l3} and $\varphi:A_1\rightarrow \mathrm{Aut}(H_1)$. Then $\Gamma_1\times \Gamma_2$ is isomorphic to the semidirect product $(\Gamma_2\times H_1)\rtimes_{\overline{\varphi}} A_1$, where the action of $A_1$ on $\Gamma_2\times H_1$ is trivial on the first factor and is given by $\varphi$ on the second. One can conclude by induction, as  $|\Gamma_2\times H_1|<|\Gamma_1\times \Gamma_2|$.

So, by the above cited results,  if $\Gamma$ satisfies one of the three conditions \eqref{l1}, \eqref{l2} or \eqref{l3}, for every $n$, for every finite subset of primes $\mathcal{T}$ of $K$ above some primes in $\mathcal{S}$ and setting, for every $v\in \mathcal T$,  $\Gamma_{v}^n=1$ equal to the trivial group, the Grunwald problem for $(\Gamma^n,\mathcal{T}, \{1\}_{v\in \mathcal{T}})$ has a solution $M_n$.  
Thus, by Galois theory, for every fixed finite extension $F/K$, if $n$ is big enough, we can choose a subextension of $M_n$ giving a solution to the Grunwald problem for $(\Gamma,\mathcal{T}, \{1\}_{v\in \mathcal{T}})$ linearly disjoint from $F$ over $K$.
In particular, if $\mathcal{S}=\{p_1,p_2,\ldots\}$ and $\mathcal{T}_m=\{v_1,\ldots,v_m\}$ are primes of $K$, with $v_i$ above $p_i$ for all $i$'s, there exists a family of extensions of $K$, $\{L_m\}_m$, such that $L_m/K$ is a solution of the Grunwald problem $(G_m,\mathcal{T}_m, \{1\}_{v\in \mathcal{T}_m})$ and $L_m$ is linearly disjoint over $K$ from the compositum of $L_1,\ldots, L_{m-1}$. The compositum of all such $L_m$'s is then a realisation of $G$ over $K$ with the required properties.
\end{proof}
\begin{remark}
In \cite[Theorem 1.2 and Corollary 1.5]{DG} D\`ebes and Ghazi proved that if $G$ is a finite group which is regular over $K$ (i.e. there exists a Galois extension $E/K(T)$ of group $G$ such that $E \cap \overline{K}=K$), and if the set $\mathcal{S}$ does not intersect  a finite set of 'bad primes' for the extension $E/K(T)$ (see \cite{DG}, Theorem 1.2, condition \emph{(good red)} for a precise definition), then every unramified Grunwald problem has a solution. More precisely, if $v$ is a prime of $K$ above some prime in $\mathcal{S}$, for every family of unramified extensions $F_v/K_v$ of group embeddable in $G$, there is a Galois extension $F/K$ of group $G$ having these local fields as completions. 

Via this result and using Theorem \ref{main1}, Theorem \ref{T1} can be extended to realisations of direct products of groups $\{G_m\}_m$ such that, for every $m$:
\begin{enumerate}[(i)]
\item\label{cc1} $G_m$ is regular over $K$;
\item\label{cc2} the set $\mathcal{S}$ does not intersect a finite set of 'bad' primes for $G_m$ (in particular for every $p\in\mathcal{S}$, $p \nmid |G_m|$).
\end{enumerate}
Notice that, if the answer to the \emph{Regular Inverse Galois Problem} is affirmative, then condition \eqref{cc1} holds for any finite group. 
So conjecturally, Theorem \ref{T1} should be true for any direct product of finite groups of unbounded exponents, provided that the set $\mathcal{S}$ satisfies condition \eqref{cc2}.
\end{remark}
\subsection{Some remarks related to property (B)}
We end the section with some speculations on the link between the boundedness of the local degrees for a Galois extension and property (B) for groups, whose definition is recalled in the Introduction.

As already said, in \cite{ADZ} property (B) is shown  to be true for groups $G$ whose inner automorphism group has finite exponent and a natural question is to understand whether there are other families of groups with (B).  To this aim, we may relax Question \ref{Q1} and consider only one direction. More precisely, let $G$ be a profinite group, $K$ be a number field and $L/K$ a Galois extension of group $G$:
\begin{question}\label{Q2} Does there exist a group theoretical property $P(p,B_p)$ on  $G$ which implies that the field $L$ has local degrees at the prime $p$ bounded by some positive constant $B_p$?
\end{question}
Clearly a positive answer to Question \ref{Q2} would give a new family of groups with (B).

Notice that $G$ satisfies property $P(p,B)$ for every $p$ and for some constant $B$ (independent of $p$) if and only if $G$ has bounded exponent. So, for instance, a na\"ive guess for property $P(p,B_p)$ could be that $\exp(G)$ is divisible by a finite number of primes (depending on $B_p$) and it has finite $p$-adic valuation. However, this cannot be a good formulation; indeed, if $L$ is the compositum of all fields of the form $\Q(\zeta_{2^n})$ and $G=\mathrm{Gal}(L/\Q)$, since $[\Q(\zeta_{2^n}):\Q]=\phi(2^n)=2^{n-1}$, then $G$ is a $2$-group and  the $p$-adic valuation of $\exp(G)$ is 0 for all $p\neq 2$. However $L$ has infinite local degrees above all rational primes, since it contains infinitely many roots of unity (see also Remark \ref{rem_mono} in Section \ref{Section3}).

Theorem \ref{main1} implies that Question \ref{Q2} has a negative answer when $G$ is an infinite direct product of families of groups for which certain Grunwald problems have 'enough' solutions. It is reasonable to believe that the answer to Question \ref{Q2} should be negative for all profinite groups of unbounded exponent, under the assumption of a certain flexibility on the local structure of the solutions of general embedding problems (a problem which, at the current state, seems to be out of reach).
\begin{remark}
Even though without implications on Property (B), it is of some interest to consider also the reverse direction in Question \ref{Q1}, namely: if $K$ is a number field and $L/K$ is a Galois extension of group $G$ with bounded local degrees at all primes, what can be said on the structure of $G$?

A first answer is that $G$ is not torsion free. Assume the contrary, then since $L/K$ is non trivial, there exists at least one prime $v$ of $K$ which ramifies in $L/K$. But if $G$ is torsion free, all decomposition subgroups are either trivial or infinite. Therefore, the decomposition subgroups at $v$ are infinite and the local degrees at $v$ are not bounded. This last also implies the local degrees of the maximal totally real extension of the rationals, $\Q^{tr}/\Q$, cannot be bounded at all primes, as it is proved in \cite{ADZ} that $\mathrm{Gal}(\Q^{tr}/\Q)$ is torsion free (actually, they are unbounded at all primes, as $\Q^{tr}$ contains the splitting fields $K_p$ of the $p$-Eisenstein polynomials $x^3-2px+p$, whose discriminant is $p^2(32p-27)$, and $p$ totally ramifies in $K_p$ for all $p$).

Moreover, it is easy to see that if $L/K$ is Galois, then it has bounded local degrees at $p$ if and only if there exists a finite extension $F/K$, with $F\subseteq L$, such that if $\mathfrak{S}$ is the set of primes of $F$ above $p$, then $L\subseteq F^{tot,\mathfrak{S}}$, where $F^{tot,\mathfrak{S}}$ is the maximal extension of $F$ is which every prime of $\mathfrak{S}$ splits totally. 
This implies in particular, that $G$ has a finite index subgroup which is a quotient of $\mathrm{Gal}(F^{\mathfrak{S}}/F)$. However, while the structure of the absolute Galois group of $F^{\mathfrak{S}}$ has been described by Pop \cite{Pop96} (see also the generalisation by Haran, Jarden and Pop in \cite{HJP}) as a free product of local groups, to the author's knowledge no much is known about the structure of $\mathrm{Gal}(F^{\mathfrak{S}}/F)$, not even whether an analogue result to the one of \cite{ADZ} for $\mathrm{Gal}(\Q^{tr}/\Q)$ holds.
\end{remark}
\section{Monogenic extensions and an explicit example}\label{Section3}
In this section we will construct explicitly two realisations over $\Q$, one with bounded local degrees at all primes, and one with unbounded local degrees at all primes, of a certain infinite direct product of abelian groups of even orders (a case not covered  by Theorem \ref{T1}). 
\subsection{Local degrees and monogenic extensions}
We start with a result concerning the link between the boundedness of the local degrees for a field $L$ and the boundedness of the degrees of its finite monogenic subextensions. 
\begin{proposition}\label{finiteroot}
Let $L$ be an infinite algebraic extension of $\Q$ with bounded local degrees at the rational prime $p$. Let $\{\alpha_m\}_m$ be a family of algebraic integers in $L$ of degrees tending to infinity.  Set $L_m=\mathbb{Q}(\alpha_m)$ and denote by $\mathcal{O}_{L_m}$ its ring of integers. Then $p$ divides the index $[\mathcal{O}_{L_m}:\mathbb{Z}[\alpha_m]]$ for all but finitely many $m$.

In particular, if the local degrees of $L$ at $p$ are bounded by $B$, all the finite monogenic subextensions of $L$ have degree at most $p^{B^2+1}B^2$.
\end{proposition}
\begin{proof}
Suppose that the index $[\mathcal{O}_{L_m}:\mathbb{Z}[\alpha_m]]$ is prime to $p$ for infinitely many $m$. For every $m$, let $f_m(x)\in \mathbb{Z}[x]$ be the minimal polynomial of $\alpha_m$. Then, by Dedekind's criterion, for infinitely many $m$'s the splitting of $p$ in $\mathcal{O}_{L_m}$ is determined by the factorisation of $f_m(x)$ mod $p$.
More precisely, let \[\overline{f_m(x)}=\overline{g_{m_1}(x)}^{e_{m_1}}\cdots \overline{g_{m_n}(x)}^{e_{m_n}}\] be the factorisation into coprime factors and set $f_{m_j}=\deg(\overline{g_{m_j}(x)})$. Then the integers $e_{m_i}$'s and the $f_{m_i}$'s are, respectively, the ramification indexes and inertia degrees of the primes of $\mathcal{O}_{L_m}$ above $p$.
By assumption, for all $m$ and for all $i$, $e_{m_i}$ and $f_{m_i}$ are both bounded by a constant. Thus the number of possible factors $\overline{g_{m_j}(x)}$, and hence the number of possible polynomials $\overline{f_m(x)}$, is finite. In particular, for all $m$, the degree of $\overline{f_m(x)}$ is bounded. 
On the other hand, we have $\deg \overline{f_m(x)}=\deg f_m(x)$, contradicting the fact that $\{\alpha_m\}_m$ is a family of algebraic integers of degree tending to infinity.

In particular, if $B$ is a bound for  the local degrees of $L$ at $p$ and if $\alpha_m$ generates a monogenic extension inside $L$, then, the coprime factors of $f_m(x)$ modulo $p$ have degree at most $B^2$ and so the number of such factors is clearly at most $p^{B^2+1}$.  Therefore the degree of $f_m(x)$ is bounded by $p^{B^2+1}B^2$.
\end{proof}
\begin{remark}\label{rem_mono}
Proposition \ref{finiteroot} implies, in particular, that:
\begin{enumerate}[(i)]
\item $L$ contains finitely many roots of unity, as cyclotomic fields are monogenic (the finiteness of the roots of unity in $L$ can also be proved by direct computations on the ramification of primes in cyclotomic extensions);
\item all roots of $p$-Eisenstein polynomials in $L$ are algebraic elements of bounded degree (this follows from a well-known result \cite[Section 2.5, Exercise 19]{BS}).
\end{enumerate}
\end{remark}

\subsection{A more explicit construction}
Consider the set \[\Lambda=\{q \text{ prime}\ \mid q-1 \text{ is square free}\}.\] 
By \cite{Mi} (see also \cite[Section 6.2.2]{HM}) this is an infinite set.

For every $n$, we denote by $C_n$ the cyclic group of order $n$.
We have then the following:
\begin{theorem}\label{expl-real}
Let $\mathcal{S}$ be a set of rational primes. 
Let
$$G=\prod_{q\in \Lambda} C_{q-1}.$$  
Then $G$  admits two realisations over $\Q$, one with infinite local degrees at all the primes in $\mathcal{S}$ (and actually, at all the rational primes)  and one with bounded local degrees at all the primes in $\mathcal{S}$.
\end{theorem}
\begin{proof}
\emph{A realisation with infinite local degrees.}
Let $L$ be the compositum of all fields $\{\Q(\zeta_{q})\}_{q\in \Lambda}$, where $\zeta_q$ is a primitive $q$-th root of unity. Clearly $L$ is the compositum of fields which are all pairwise linearly disjoint over $\Q$ and so $\mathrm{Gal}(L/\Q)=G$. Moreover, in view of Remark \ref{rem_mono}, $L$ has infinite local degree at every rational prime, since it contains infinitely many roots of unity.

\emph{A realisation with bounded local degrees.}  Fix a prime number $q$, a number field $F$ and a finite set of rational primes $\mathcal{T}$. We first show how to construct explicitly a Galois extension of $L/\Q$ cyclic of degree $q$ such that $L\cap F=\Q$ and such that all primes in $\mathcal{T}$ totally split in $L$.

Suppose that $|\mathcal{T}|=n$. 
Dirichlet's theorem on primes in arithmetic progressions ensures that there exists a set of primes $\mathcal{L}=\{\ell_1, \cdots , \ell_{n+1}\}$ such that for every $i=1,\ldots,n+1$ we have $\ell_i$ is congruent to 1 modulo $q$, $\ell_i\not\in\mathcal{T}$ and $\Q(\zeta_{\ell_i})\cap F=\Q$, where $\zeta_{\ell_i}$ is a primitive $\ell_i$-th root of unity.

Clearly the field $\Q(\zeta_{\ell_1\cdots\ell_{n+1}})$ contains a subfield $E$ such that $E/\Q$ is Galois with Galois group $q$-elementary abelian of order $q^{n+1}$. In particular, the extension $E/\Q$ is unramified outside the set of primes $\ell_1, \ldots , \ell_{n+1}$, so in particular at the primes in $\mathcal{T}$.

Consider the subfield of $E$ fixed by all elements $\{\sigma_p\mid\ p\in \mathcal{T}\}$ where $\sigma_p$ denotes the Frobenius at $p$ (which generates the cyclic decomposition group at $p$). This extension has degree at least $q$ over $\Q$ and,  in particular,  contains a cyclic subfield $L$ of degree $q$ over $\Q$. We remark that, by construction, $L\cap F=\Q$  and every prime $p\in\mathcal{T}$ splits totally in $L/\Q$.

Now, write $\mathcal{S}=\{p_1, p_2,\ldots\}$ and $\Lambda=\{q_1,q_2,\ldots\}$ with $p_j< p_i$ and $q_i<q_j$ for $i< j$. We have $q_i-1=\gamma_{i,1}\cdots\gamma_{i,n_i}$ where $\gamma_{i,j}$ are pairwise distinct primes and $C_{q_i-1}=\prod_{j=1}^{n_i} C_{\gamma_{i,j}}$.

By the above result, we can construct inductively, for every $i\geq 1$ and for every $j=1,\ldots,n_i$, a realisation $L_{i,j}/\Q$ of the group $C_{\gamma_{i,j}}$ over $\Q$ which is linearly disjoint from any $L_{m,k}$ with $m<i$ or $k<j$ and in which all the primes $p_1,\ldots,p_i$ split totally. The compositum $L/\Q$ of all the fields $L_{i,j}$ is therefore a Galois extension of group $G$ having local degrees at the prime $p_n\in \mathcal{S}$ bounded by $\prod_{k=1}^n (q_k-1)$.
\end{proof}

\section{A construction with iterated semi-direct products}\label{Section4}
In this section we show the existence of realisations with bounded local degrees at some primes for certain profinite groups which are not necessarily direct products.
The construction here is based on the existence of solutions of certain embedding problems with special local properties.

Let $\Gamma$ and $G$ be finite groups with a surjective homomorphism $\varphi:\Gamma \twoheadrightarrow G$ and set $H=\ker(\varphi)$. Let $F/K$ be a normal extension of number fields with Galois group $G$. 

The embedding problem for $(F/K, \Gamma, H)$ asks whether there exists a finite Galois extension $L/K$ of group $\Gamma$ such that $F\subseteq L$ and for every $\sigma\in \mathrm{Gal}(L/K)$, we have $\sigma\restriction_{F}=\varphi(\sigma)$. If this exists, it is called a solution of the embedding problem and $H$ is called the kernel of the embedding problem.

In \cite{Sha2} (see also \cite[Theorem 9.5.10]{NSW}) Shafarevich proves that every \emph{split} embedding problem with nilpotent kernel has a solution (this is the fundamental step in the proof of  the realisability of  finite solvable groups over number fields). More precisely, he proves that if $G$ and $H$ are finite groups with $H$ nilpotent and $G$ acting on $H$,  $\Gamma=H\rtimes G$ and $F/K$ is a finite Galois extensions of number fields of group $G$, then the embedding problem for $(F/K,\Gamma,H)$ has a solution  $L/K$.

The solution $L/K$ obtained has however special local properties and using this approach it is not possible to realise any given local behaviour. In particular, our construction is based on the fact that all primes ramifying in $F/K$ split totally in $L/F$ (see \cite[Theorem 9.5.11 (i)] {NSW}). 

We have the following:
\begin{proposition}\label{t31} Let $K$ be a number field and denote by $\mu(K)$ the group of roots of unity in $K$. 
Let $\{G_m\}_m$ be a family of finite groups such that $G_1$ is solvable and, for every $m\geq 2$, we have $G_m=N_m\rtimes G_{m-1}$ where $N_m$ is nilpotent.  Let $G$ be the inverse limit of the family  $\{G_m\}_m$ (the bonding maps being the natural projection from $G_m$ to $G_{m-1}$).
 Then:
\begin{enumerate}
\item\label{C1} for every positive integer $B$, there exists a Galois extension $L/K$ of Galois group $G$ having local degrees at at least one prime $p>B$ bounded by $|G_1|[K:\Q]$.
\item\label{C3}  Suppose that, for some $m\geq 1$, $G_m$ has a metacyclic $p$-subgroup and $\exp(G_m)$ is prime to $|\mu(K)|$. Then then there exists a Galois extension $L/K$ of Galois group $G$ having local degrees at $p$ bounded by $|G_m|[K:\Q]$.
\end{enumerate}

 \end{proposition}
\begin{proof}
{\bf Proof of \eqref{C1}.}
We start with the group $G_1$. Since this group is solvable, we can construct infinitely many linearly disjoint extensions of $K$ having $G_1$ as Galois group; this follows from the fact that $G_1^n$ is solvable for every $n$, therefore realisable by Shafarevich's theorem.

By Hermite's theorem, there are finitely many number fields with bounded discriminant and, for every prime $p$, the $p$-adic valuation of the discriminant of a number field is bounded solely in terms of the degree of the field (see for instance \cite{BG}, Theorem B.2.12). Therefore, for every $B>0$ we can choose a realisation $K_1/K$ of $G_1$ such that $\mathrm{disc}(K_1/\Q)$ is divisible by a prime $p>\max(B, \mathrm{disc}(K/\Q))$.

Let $v$ be a prime of $K$ above $p$. Then $v$ ramifies in $K_1/K$, since $p$ ramifies in $K_1/\Q$ but not in $K/\Q$ (remember that $p>\mathrm{disc}(K/\Q)$).

From \cite[Theorem 9.5.10 and Theorem 9.5.11(i)]{NSW}, for every $m\geq 2$, as $G_m=N_m\rtimes G_{m-1}$ and $N_m$ is nilpotent, there is a solution $K_m/K$ to the embedding problem $(K_{m-1}/K, G_m,N_m)$
such that the primes ramifying in $K_{m-1}/K$ split totally in $K_{m}/K_{m-1}$. In particular, by the previous considerations, $v$ ramifies in $K_m/K$ for every $m$ (it ramifies already in $K_1/K$) and therefore it splits completely in all extensions of the form $K_{m+1}/K_{m}$.

Set $L=\bigcup_m K_m$. By construction, $\mathrm{Gal}(L/K)=G$ and the local degrees of $L$ at $p$ are bounded by $[K_1:\Q]$.

\smallskip

{\bf Proof of \eqref{C3}.} 
Let $p$ be a rational prime and suppose that there exists an index $m$ such that $G_m$ possesses a $p$-metacyclic subgroup, say $H_{m,p}$ and $\exp(G_m)$ is prime to $|\mu(K)|$.

In view of Lemma 2.10 of \cite{Neftin}, the group $H_{m,p}$ is realisable over any $p$-adic field. Moreover, from Neukirch's result \cite{Neuk}, we can find a realisation $K_m/K$ of $G_m$ such that, for every prime $v$ of $K_m$ above $p$, the local extension $K_{m,v}/K_p$ has Galois group $H_{m,p}$. This implies that $p$ ramifies in $K_m$, otherwise the local extension would be cyclic.

Again, by \cite[Theorem 9.5.10 and Theorem 9.5.11(i)]{NSW}, for every $n>m$ we can find a solution $K_n/K$ of the embedding problem $(K_{n-1}/K, G_n, N_m)$ such that every prime $v$ of $K_n$ above $p$ splits completely in $K_{n}/K_m$. 
Setting $L=\bigcup_n K_n$, by construction $\mathrm{Gal}(L/K)=G$ and the local degrees of $L$ at $p$ are bounded by $[K_m:\Q]$.
 \end{proof}

\begin{remark}
An example of a family $\{G_m\}_m$ satisfying the hypothesis of Proposition \ref{t31} can be easily constructed, as in Theorem \ref{T1}, by taking direct products of nilpotent groups. However, this is not the only possible example.
Indeed, let $G=\mathrm{Aut}(T_2)$ be the group of automorphisms of the rooted binary tree $T_2$.  Then $G$ is the inverse limit of the family of groups $\{G_m\}_m$ where
$G_1=\{1\}$, $G_{m+1}\simeq N_m \rtimes G_m$, where $G_m$ is the automorphism group of the subtree of depth $m$ and $N_m\simeq (\mathbb{Z}/2\mathbb{Z})^m$ (the action is the one of $G_m$ on the subtree in the next level).

However $G$ cannot be written as a direct product of a family of finite nilpotent groups. Indeed, suppose $G=\prod_n H_n$ with $H_n$ finite and nilpotent for every $n$. Then the center of $G$ would be the direct product of the centers of the $H_n$'s, thus infinite, while it consists of only two elements, the identity and the conjugation (which flips all subtrees). 
\end{remark}

\begin{remark} 
Notice that, in point \eqref{C3} of Proposition \ref{t31}, the group $H_{m,p}$ needs just to be not cyclic and realisabe over any $p$-adic field. In particular, if $K=\Q$ and $p$ is odd, we may suppose that  $H_{m,p}$  is any $p$-group with two generators: indeed, by Shafarevich's theorem \cite{Sha}, these groups are realisable over $\Q_p$. 
\end{remark}

\section*{Acknowledgments} { The author wants to thank  Francesco Amoroso for pointing out the relation between local degrees and monogenic subextensions, described in Proposition \ref{finiteroot}, Bruno Angl\`es for suggesting part of the construction in Theorem \ref{expl-real} and  Fran\c cois Dahmani for useful discussions. She also thanks Umberto Zannier and Lukas Pottmeyer for useful comments and for their interest in this article. The author's work has been funded by the ANR project Gardio 14-CE25-0015.}

 \def\cprime{$'$}
\providecommand{\bysame}{\leavevmode\hbox to3em{\hrulefill}\thinspace}
 \providecommand{\MR}{\relax\ifhmode\unskip\space\fi MR }
 \providecommand{\MRhref}[2]{
   \href{http://www.ams.org/mathscinet-getitem?mr=#1}{#2}
 }
 \providecommand{\href}[2]{#2}

\end{document}